\documentclass[11pt]{amsart}

\usepackage{amssymb,mathtools}
\usepackage{microtype}
\usepackage{aliascnt}
\usepackage[hidelinks]{hyperref}
\hypersetup{
  pdftitle={Lattice balls with large additive energy in discrete cubes},
  pdfauthor={xinyu long},
  pdfsubject={Additive energy of high-dimensional lattice balls},
  pdfkeywords={additive energy, discrete cubes, lattice points, Euclidean balls, Gram matrices}
}

\allowdisplaybreaks
\numberwithin{equation}{section}

\newtheorem{theorem}{Theorem}[section]
\newaliascnt{proposition}{theorem}
\newtheorem{proposition}[proposition]{Proposition}
\aliascntresetthe{proposition}
\newaliascnt{lemma}{theorem}
\newtheorem{lemma}[lemma]{Lemma}
\aliascntresetthe{lemma}
\newaliascnt{corollary}{theorem}
\newtheorem{corollary}[corollary]{Corollary}
\aliascntresetthe{corollary}
\theoremstyle{remark}
\newaliascnt{remark}{theorem}
\newtheorem{remark}[remark]{Remark}
\aliascntresetthe{remark}

\newcommand{\R}{\mathbb R}
\newcommand{\Z}{\mathbb Z}
\newcommand{\one}{\mathbf 1}
\newcommand{\Ecal}{\mathcal E}
\newcommand{\Sym}{\operatorname{Sym}}
\newcommand{\tr}{\operatorname{tr}}
\newcommand{\vol}{\operatorname{vol}}

\title{Lattice balls with large additive energy in discrete cubes}
\author{xinyu long}
\address{School of Mathematics, Shandong University\\ Jinan, Shandong, China}
\email{longxinyu@sdu.edu.cn}


\begin{document}

\begin{abstract}
For a finite set $A$ in an abelian group, let
\[
 E(A)=\#\{(a_1,a_2,a_3,a_4)\in A^4:a_1+a_2=a_3+a_4\}.
\]

We obtain an estimate uniform in $d$ that compares the normalized additive energy of $\mathbb{Z}^d \cap B_d(R)$ with the continuous energy of $B_d(R)$ .
If $R_d/\sqrt d\to\infty$, then
\[
 \lim_{d\to\infty}
 \left(
  \frac{E(\Z^d\cap B_d(R_d))}
       {|\Z^d\cap B_d(R_d)|^3}
 \right)^{1/d}
 =\frac{4\sqrt3}{9}.
\]
As an application, consider
\[
 A_n=R_n\one_d+\bigl(\Z^d\cap B_d(R_n)\bigr),
\]
where $d=d(n)\to\infty$ satisfy $\log d=o(\log n)$, and $R_n=\lfloor(n-1)/2\rfloor$.
Then $A_n\subset\{0,1,\ldots,n-1\}^d$ and
\[
 \log E(A_n)
 =3\log|A_n|-d\log\frac{3\sqrt3}{4}+o(d).
\]
In particular, taking $d=\lfloor(\log n)^{1/2}\rfloor$ gives an explicit construction answering a question of Shao \cite{Shao2026}. We also prove that in Gram-matrix coordinates, the exponential rate of the continuous ball energy is determined by a fixed dimensional determinant maximization whose extremizer is the Gram matrix of a regular tetrahedron.
\end{abstract}

\maketitle

\section{Introduction}

Let $A\subset G$ be a finite subset of an abelian group $G$. The additive energy of $A$ is defined by
\[
 E(A)=\#\{(a_1,a_2,a_3,a_4)\in A^4:a_1+a_2=a_3+a_4\}.
\]
Equivalently,
\[
 E(A)=\|\one_A*\one_A\|_2^2.
\]
Thus $E(A)$ measures the concentration of the representation function of the sumset $A+A$. In particular, Cauchy--Schwarz gives
\[
 E(A)\ge \frac{|A|^4}{|A+A|},
\]
so small doubling implies large additive energy. A partial converse is supplied by the Balog--Szemer\'edi--Gowers theorem: if $E(A)$ is comparable with the trivial upper bound $|A|^3$, then a large subset of $A$ has small doubling. Combined with Freiman-type inverse theorems, this places a large part of $A$ inside a generalized arithmetic progression of controlled rank and size; see, for example, \cite[Chapters~5 and~6]{TaoVu2006}. It is interesting that how much additive energy is compatible with additional geometric restrictions on the ambient set.

For $n\ge2$, Shao \cite{Shao2026} defines $t_n$ to be the smallest real number such that
\[
 E(A)\le |A|^{t_n}
\]
for every positive integer $d$ and every set
$A\subset\{0,1,\ldots,n-1\}^d$. For the binary cube, Kane and Tao \cite[Theorem~7]{KaneTao2017} proved that $t_2=\log_2 6.$ De Dios Pont, Greenfeld, Ivanisvili, and Madrid \cite{deDiosPontEtAl2023} developed a framework covering additive energies on discrete cubes, including higher energies and discrete extension inequalities. Their work also shows that the full product set need not be extremal once the alphabet has more than two elements, in particular,
\[
 t_3\ge 2\log_2(2.5664)>\log_3 19.
\]

Kova\v{c} \cite{Kovac2023} later proved the remaining inequality for binomial sums, thereby completing the sharp estimate for higher additive energies on the binary cube. More recently, Beker, Crmari\'c, and Kova\v{c} \cite{BekerCrmaricKovac2025} established analogous estimates for Gowers norms on discrete cubes, obtaining results for the binary cube and, for each fixed Gowers order, two-sided asymptotic estimates as the alphabet size tends to infinity.


For general $n$, Shao proves that, for some absolute constant $c>0$,
\[
 3-\bigl(1+o_{n\to\infty}(1)\bigr)
       \log_n\frac{3\sqrt3}{4}
 \le t_n\le 3-\log_n(1+c),
\]
where $\log_n x=(\log x)/(\log n)$.  In \cite[Question~2.4]{Shao2026}, Shao asks for an explicit set attaining the same exponent and suggests lattice points in a high-dimensional Euclidean ball as a candidate.In this paper, we verify this proposal and obtain an error bound that is uniform in the dimension.  

For $d\ge1$ and $R>0$, let $B_d(R)=\{x\in\R^d:\|x\|_2\le R\},$ $B_d=B_d(1),$ $v_d=\vol(B_d),$ then set $\Lambda_{d,R}=\Z^d\cap B_d(R).$ For a measurable set $S\subset\R^d$ of finite measure, define its continuous additive energy by
\[
 \Ecal(S)
 =\int_{(\R^d)^3}
   \one_S(x_1)\one_S(x_2)\one_S(x_3)
   \one_S(x_1+x_2-x_3)
   \,dx_1\,dx_2\,dx_3,
\]
for $e_d=\frac{\Ecal(B_d)}{v_d^3},$ Shao \cite{Shao2014} proved a continuous rearrangement theorem showing, in particular, that among compact sets of prescribed volume, $\Ecal(S)$ is maximized by Euclidean balls. Moreover, the normalized ball energy satisfies $e_d^{1/d}\longrightarrow \frac{4\sqrt{3}}{9}$, which is recorded in \cite{Shao2014}, computed in detail in \cite[Section~3.1]{Mazur2016}, and used in \cite[Section~2]{Shao2026} to motivate the lower bound for $t_n$.


The same constant also arises from the Hausdorff–Young inequality in Babenko \cite{Babenko1961} and Beckner \cite{Beckner1975}, with the Fourier transform normalized by

\[
 \widehat f(\xi)=\int_{\R^d}f(x)e^{-2\pi i x\cdot\xi}\,dx,
\]
at exponents $4/3$ and $4$, we have
\[
 \|\widehat f\|_4^4
 \le \left(\frac{4\sqrt3}{9}\right)^d\|f\|_{4/3}^4
 \qquad(f\in L^{4/3}(\R^d)).
\]
Equivalent forms of Young's convolution inequality were obtained by Brascamp and Lieb \cite{BrascampLieb1976} and Fournier \cite{Fournier1977}. Since $\Ecal(S)=\|\widehat{\one_S}\|_4^4$, applied to indicator functions, these inequalities yield $4\sqrt3/9$. The asymptotic for $B_d$ shows that the indicators of Euclidean balls attain the same exponential rate as $d \to \infty$.

For each fixed $d$, a standard lattice approximation shows that the normalized energy of $\Lambda_{d,R}$ converges to $e_d$ as $R\to\infty$. This fixed dimensional case, however, is insufficient for the discrete cube problem, where $d=d(n)$ and $R\asymp n$ tend to infinity simultaneously. In the present paper, by tiling $\R^d$ with unit cubes, we obtain a uniform estimate with logarithmic error $O(d^{3/2}/R)$.

\begin{theorem}\label{thm:lattice-ball}
Let $R_d>0$ satisfy $\frac{R_d}{\sqrt d}\longrightarrow\infty.$ Then
\[
 \lim_{d\to\infty}
 \left(
  \frac{E(\Lambda_{d,R_d})}{|\Lambda_{d,R_d}|^3}
 \right)^{1/d}
 =\frac{4\sqrt3}{9}.
\]
\end{theorem}

The following consequence gives a two-sided logarithmic asymptotic and therefore implies the lower bound requested in \cite[Question~2.4]{Shao2026}. 

\begin{corollary}\label{cor:construction}
Let $d = d(n) \to \infty$ with $\log d = o(\log n)$, and let $R_n=\left\lfloor\frac{n-1}{2}\right\rfloor,$ $A_n=R_n\one_d+\Lambda_{d,R_n},$
where $\one_d=(1,\ldots,1)\in\R^d$. Then $A_n \subset \{0,1,\dots,n-1\}^d$, and
\begin{equation}
 \label{eq:log-An}
 \log |A_n| = d \log n - \frac{d}{2} \log d + O(d),
\end{equation}
\begin{equation}
 \label{eq:log-EAn}
 \log E(A_n) = 3 \log |A_n| - d \log \frac{3\sqrt{3}}{4} + o(d).
\end{equation}

Consequently,
\[
 E(A_n)
 =|A_n|^{\,3-\frac{\log(3\sqrt3/4)+o(1)}{\log n}}
 =|A_n|^{\,3-(1+o(1))\log_n(3\sqrt3/4)}.
\]
In particular, one may take $d=\left\lfloor(\log n)^{1/2}\right\rfloor.$
\end{corollary}

After passing to the Gram matrix of three vectors, the continuous energy is expressed as an integral over a six-dimensional compact set. A finite dimensional Laplace principle then reduces the asymptotic to a determinant maximization, whose maximizer corresponds to a regular tetrahedron.


The paper is organized as follows.  Section~\ref{sec:continuous} establishes the continuous asymptotic, and Section~\ref{sec:lattice} proves the main results.

\section{The continuous energy of a Euclidean ball}\label{sec:continuous}

\subsection{Gram matrix coordinates}

Let $\Sym_3^{++}$ denote the cone of positive definite real symmetric $3\times3$ matrices. We write $dG=\prod_{1\le i\le j\le3}dg_{ij}$ for Lebesgue measure on the six-dimensional vector space of real symmetric $3\times3$ matrices.

\begin{lemma}\label{lem:gram}
Let $d\ge3$. There is a constant $C_d>0$ such that, for every nonnegative measurable function $F$ on $\Sym_3^{++}$,
\begin{equation}
 \label{eq:gram-int}
 \int_{\R^{d\times3}}F(X^{\mathsf T}X)\,dX
 =C_d\int_{\Sym_3^{++}}
 F(G)(\det G)^{(d-4)/2}\,dG.
\end{equation}
\end{lemma}

\begin{proof}
The set of matrices $X\in\R^{d\times3}$ of rank less than three has Lebesgue measure zero. For a full-rank matrix, write the QR decomposition as
\[
 X=QR,
\]
where $Q=(q_1,q_2,q_3)$ has orthonormal columns and
\[
 R=
 \begin{pmatrix}
  r_{11}&r_{12}&r_{13}\\
  0&r_{22}&r_{23}\\
  0&0&r_{33}
 \end{pmatrix},
 \qquad r_{11},r_{22},r_{33}>0.
\]
Applying polar coordinates successively to the three columns gives
\[
 dX
 =C_d' r_{11}^{d-1}r_{22}^{d-2}r_{33}^{d-3}
 \,dR\,d\mu(Q),
\]
where $d\mu$ is the invariant measure on the Stiefel manifold of orthonormal $3$ frames and $C_d'>0$ depends only on $d$.

For $G=R^{\mathsf T}R$. In the order $(r_{11},r_{12},r_{13},r_{22},r_{23},r_{33}),$ a direct calculation gives
\[
 \left|
 \frac{\partial(g_{11},g_{12},g_{13},g_{22},g_{23},g_{33})}
      {\partial(r_{11},r_{12},r_{13},r_{22},r_{23},r_{33})}
 \right|
 =8r_{11}^3r_{22}^2r_{33}.
\]
Since
\[
 \det G=(r_{11}r_{22}r_{33})^2,
\]
we obtain
\[
 r_{11}^{d-1}r_{22}^{d-2}r_{33}^{d-3}\,dR
 =\frac18(\det G)^{(d-4)/2}\,dG.
\]
The integrand is independent of $Q$. Integrating over the Stiefel manifold and absorbing the resulting factor into $C_d$ proves the \eqref{eq:gram-int}.
\end{proof}

\subsection{Finite dimensional Laplace principle}

\begin{lemma}\label{lem:laplace}
Let $K\subset\R^m$ be compact and equal to the closure of its interior. Let $f:K\to[0,\infty)$ be continuous, and suppose that $\max_{x\in K}f(x)>0.$ Then
\[
 \lim_{a\to\infty}
 \left(\int_K f(x)^a\,dx\right)^{1/a}
 =\max_{x\in K}f(x).
\]
\end{lemma}

\begin{proof}
The upper bound follows from
\[
 \int_K f(x)^a\,dx\le |K|(\max_{x\in K}f(x))^a.
\]
For the lower bound, fix $0<\delta<(\max_{x\in K}f(x))/2$. Since $K$ is the closure of its interior, there is an interior point $x_\delta$ such that
\[
 f(x_\delta)>\max_{x\in K}f(x)-\delta.
\]
By continuity, there is an open set $U_\delta\subset K$ of positive measure on which $f>\max_{x\in K}f(x)-2\delta$. Hence
\[
 \int_K f(x)^a\,dx
 \ge |U_\delta|(\max_{x\in K}f(x)-2\delta)^a.
\]
Taking $a$-th roots, passing to liminf, and then letting $\delta\downarrow0$ completes the proof.
\end{proof}

\subsection{Determinant maximization}

Let $u=(1,1,1)^{\mathsf T}$ and define
\[
 \mathcal D
 =\{G\in\Sym_3:G\succeq0,\ g_{11},g_{22},g_{33}\le1\},
\]
\[
 \mathcal K
 =\{G\in\mathcal D:u^{\mathsf T}Gu\le1\}.
\]
Both sets are compact. Indeed, positive semidefiniteness gives
$|g_{ij}|\le\sqrt{g_{ii}g_{jj}}\le1$. $\mathcal D$, $\mathcal K$ are also the closures of their interiors: if $G$ belongs to either set, then
\[
 (1-\delta)G+\eta I_3,
 \qquad 0<3\eta<\delta,
\]
is positive definite and satisfies all the defining inequalities strictly, and it converges to $G$ as $\delta,\eta\to0$.

\begin{proposition}\label{prop:determinants}
We have
\begin{equation}
 \label{eq:det-D}
 \max_{G\in\mathcal D}\det G=1
\end{equation}
and
\begin{equation}
 \label{eq:det-K}
 \max_{G\in\mathcal K}\det G=\frac{16}{27}.
\end{equation}
The maximum in \eqref{eq:det-K} is attained at
\[
 G_*
 =\begin{pmatrix}
  1&-1/3&-1/3\\
  -1/3&1&-1/3\\
  -1/3&-1/3&1
 \end{pmatrix}.
\]
\end{proposition}

\begin{proof}
For $G\in\mathcal D$, Hadamard's determinant inequality gives
\[
 \det G\le g_{11}g_{22}g_{33}\le1,
\]
with equality at $G=I_3$.

We now consider $\mathcal K$. The matrix $G_*$ is positive definite: its eigenvalue in the direction of $u$ is $1/3$, and its two eigenvalues on $u^\perp$ are $4/3$. Moreover, $u^{\mathsf T}G_*u=1,$ then
\[
 \det G_*=\frac13\left(\frac43\right)^2=\frac{16}{27}.
\]

It remains to prove the upper bound. If $G\in\mathcal K$ is singular, then $\det G=0$. We may therefore assume that $G$ is positive definite. For each $\sigma\in S_3$, let $P_\sigma$ be the corresponding permutation matrix and let
\[
 \overline G
 =\frac16\sum_{\sigma\in S_3}P_\sigma GP_\sigma^{\mathsf T}.
\]
The set $\mathcal K$ is convex and invariant under simultaneous permutations of rows and columns, so $\overline G\in\mathcal K$. The function $G\mapsto\log\det G$ is concave on the positive definite cone. Indeed, if $H(t)=(1-t)A+tB$, then
\[
 \frac{d^2}{dt^2}\log\det H(t)
 =-\tr\bigl(H(t)^{-1}(B-A)H(t)^{-1}(B-A)\bigr)\le0.
\]
It follows from Jensen's inequality that
\[
 \det\overline G\ge\det G.
\]
Thus it is enough to consider a permutation-invariant matrix
\[
 \overline G
 =\begin{pmatrix}
  a&b&b\\
  b&a&b\\
  b&b&a
 \end{pmatrix}.
\]
Let $s=a+2b$. The eigenvalue in the direction of $u$ is $s$, and the other two eigenvalues are
\[
 a-b=\frac{3a-s}{2}.
\]
Since $\overline G\in\mathcal K$ is positive definite, $0<s\le\frac13,$ $a\le1,$ $3a>s.$ Moreover,
\[
 \det\overline G
 =s\left(\frac{3a-s}{2}\right)^2.
\]
For fixed $s$, this expression is increasing in $a$ on the admissible range. Hence
\[
 \det\overline G
 \le\frac{s(3-s)^2}{4}.
\]
The function on the right is increasing for $0<s\le1/3$, since
\[
 \frac{d}{ds}\frac{s(3-s)^2}{4}
 =\frac{3(3-s)(1-s)}4>0.
\]
Consequently,
\[
 \det G\le\det\overline G
 \le\frac{(1/3)(3-1/3)^2}{4}
 =\frac{16}{27}.
\]
\end{proof}

\begin{remark}\label{rem:tetrahedron}
The matrix $G_*$ has the following geometric interpretation. Let $x_1,x_2,x_3$ have Gram matrix $G_*$ and let
\[
 x_4=-(x_1+x_2+x_3).
\]
Then $\|x_i\|_2=1$ for every $1\le i\le4$, and
\[
 \langle x_i,x_j\rangle=-\frac13
 \qquad(i\ne j).
\]
Thus $x_1,x_2,x_3,x_4$ are the vertices of a regular tetrahedron centred at the origin. 
\end{remark}

\subsection{Asymptotics of the continuous energy}

\begin{proposition}\label{prop:continuous}
As $d\to\infty$,
\[
 \lim_{d\to\infty}e_d^{1/d}
 =\frac{4\sqrt3}{9}.
\]
Equivalently,
\[
 \log e_d
 =d\log\frac{4\sqrt3}{9}+o(d).
\]
\end{proposition}

\begin{proof}
Since $B_d=-B_d$, the change of variables $y_3=-x_3$ gives
\[
 \Ecal(B_d)
 =\int_{B_d^3}\one_{B_d}(x_1+x_2+y_3)
   \,dx_1\,dx_2\,dy_3.
\]
Let $X\in\R^{d\times3}$ have columns $x_1,x_2,y_3$, and put $G=X^{\mathsf T}X$. The conditions $x_1,x_2,y_3\in B_d$ are equivalent to $g_{11},g_{22},g_{33}\le1,$ while
\[
 \|x_1+x_2+y_3\|_2^2=u^{\mathsf T}Gu.
\]
For $d\ge5$, Lemma~\ref{lem:gram} gives
\[
 \Ecal(B_d)
 =C_d\int_{\mathcal K}(\det G)^{(d-4)/2}\,dG
\]
and
\[
 v_d^3
 =C_d\int_{\mathcal D}(\det G)^{(d-4)/2}\,dG.
\]
The singular boundaries of $\mathcal D$ and $\mathcal K$ have Lebesgue measure zero, so  replacing the positive definite domains by their compact closures does not change the integrals, because the singular boundaries have Lebesgue measure zero.

Let $a_d=(d-4)/2$. The constants $C_d$ cancel, and Lemma~\ref{lem:laplace} together with Proposition~\ref{prop:determinants} gives
\[
 \lim_{d\to\infty}
 \left(\int_{\mathcal D}(\det G)^{a_d}\,dG\right)^{1/a_d}=1
\]
and
\[
 \lim_{d\to\infty}
 \left(\int_{\mathcal K}(\det G)^{a_d}\,dG\right)^{1/a_d}
 =\frac{16}{27}.
\]
Since $a_d/d\to1/2$,
\[
 \lim_{d\to\infty}\frac1d\log e_d
 =\frac12\log\frac{16}{27}
 =\log\frac{4\sqrt3}{9}.
\]
\end{proof}

\section{Uniform comparison for lattice points in Euclidean balls}\label{sec:lattice}

Set $Q_d=[-1/2,1/2)^d.$ The translates $a+Q_d$, $a\in\Z^d$, form a disjoint partition of $\R^d$, and every point of $Q_d$ has Euclidean norm at most $\sqrt d/2$.

\begin{lemma}\label{lem:count}
If $R>\sqrt d/2$, then
\begin{equation}
 \label{eq:ball-count}
 v_d\left(R-\frac{\sqrt d}{2}\right)^d
 \le |\Lambda_{d,R}|
 \le
 v_d\left(R+\frac{\sqrt d}{2}\right)^d.
\end{equation}
\end{lemma}

\begin{proof}
Let
\[
 U_{d,R}=\bigcup_{a\in\Lambda_{d,R}}(a+Q_d).
\]
If $x\in B_d(R-\sqrt d/2)$ and $x\in a+Q_d$, then
\[
 \|a\|_2\le\|x\|_2+\|a-x\|_2\le R,
\]
so $a\in\Lambda_{d,R}$. Conversely, if $x\in a+Q_d$ with $a\in\Lambda_{d,R}$, then
\[
 \|x\|_2\le R+\frac{\sqrt d}{2}.
\]
Thus
\[
 B_d\left(R-\frac{\sqrt d}{2}\right)
 \subset U_{d,R}
 \subset
 B_d\left(R+\frac{\sqrt d}{2}\right).
\]
The cubes in $U_{d,R}$ are disjoint and have volume one. Taking volumes proves \eqref{eq:ball-count}.
\end{proof}

\begin{lemma}\label{lem:energy-comparison}
If $R>3\sqrt d/2$, then
\[
 \Ecal(B_d)\left(R-\frac{3\sqrt d}{2}\right)^{3d}
 \le E(\Lambda_{d,R})
 \le
 \Ecal(B_d)\left(R+\frac{3\sqrt d}{2}\right)^{3d}.
\]
\end{lemma}

\begin{proof}
Write
\[
 T_{d,R}
 =\{(a,b,c)\in(\Z^d)^3:
       a,b,c,a+b-c\in\Lambda_{d,R}\}
\]
and
\[
 U_{d,R}^{(3)}
 =\bigcup_{(a,b,c)\in T_{d,R}}
   (a+Q_d)\times(b+Q_d)\times(c+Q_d).
\]
The product cubes are disjoint and have volume one, so
\[
 \vol_{3d}(U_{d,R}^{(3)})=E(\Lambda_{d,R}).
\]
For $\rho\ge0$, define
\[
 S_{d,\rho}
 =\{(x,y,z)\in(\R^d)^3:
   \|x\|_2,\|y\|_2,\|z\|_2,\|x+y-z\|_2\le\rho\}.
\]
By scaling,
\[
 \vol_{3d}(S_{d,\rho})=\Ecal(B_d)\rho^{3d}.
\]

We claim that
\[
 S_{d,R-3\sqrt d/2}
 \subset U_{d,R}^{(3)}
 \subset S_{d,R+3\sqrt d/2}.
\]
For the first inclusion, take $(x,y,z)\in S_{d,R-3\sqrt d/2}$ and choose $a,b,c\in\Z^d$ with $x\in a+Q_d,$ $y\in b+Q_d,$ $z\in c+Q_d.$ Then
\[
 \|a-x\|_2,\|b-y\|_2,\|c-z\|_2\le\frac{\sqrt d}{2}.
\]
It follows that $\|a\|_2,\|b\|_2,\|c\|_2\le R$ and
\[
 \|a+b-c\|_2
 \le\|x+y-z\|_2
     +\|a-x\|_2+\|b-y\|_2+\|c-z\|_2
 \le R.
\]
Hence $(a,b,c)\in T_{d,R}$.

For the reverse inclusion, let $(x,y,z)$ lie in a product cube indexed by $(a,b,c)\in T_{d,R}$. Then
\[
 \|x\|_2,\|y\|_2,\|z\|_2
 \le R+\frac{\sqrt d}{2}
 \le R+\frac{3\sqrt d}{2},
\]
and
\[
 \|x+y-z\|_2
 \le\|a+b-c\|_2
     +\|x-a\|_2+\|y-b\|_2+\|z-c\|_2
 \le R+\frac{3\sqrt d}{2}.
\]
Taking volumes proves the two inequalities.
\end{proof}

\begin{proposition}\label{prop:normalized-comparison}
If $R\ge3\sqrt d$, then
\[
 \left|
 \log\frac{E(\Lambda_{d,R})}{|\Lambda_{d,R}|^3}
 -\log e_d
 \right|
 \le 12\frac{d^{3/2}}R.
\]
\end{proposition}

\begin{proof}
Combining Lemmas~\ref{lem:count} and \ref{lem:energy-comparison} gives
\[
 e_d
 \left(
  \frac{R-3\sqrt d/2}{R+\sqrt d/2}
 \right)^{3d}
 \le
 \frac{E(\Lambda_{d,R})}{|\Lambda_{d,R}|^3}
 \le
 e_d
 \left(
  \frac{R+3\sqrt d/2}{R-\sqrt d/2}
 \right)^{3d}.
\]
Put $x=\sqrt d/R$, so that $0\le x\le1/3$. Using
\[
 \log(1+u)\le u,
 \qquad
 -\log(1-u)\le2u
 \quad(0\le u\le1/2),
\]
the logarithmic error on either side is at most
\[
 12dx=12\frac{d^{3/2}}R.
\]
\end{proof}

\begin{proof}[Proof of Theorem~\ref{thm:lattice-ball}]
For all sufficiently large $d$, Proposition~\ref{prop:normalized-comparison} applies. Dividing its estimate by $d$ gives
\[
 \frac1d\log\frac{E(\Lambda_{d,R_d})}{|\Lambda_{d,R_d}|^3}
 =\frac1d\log e_d
  +O\left(\frac{\sqrt d}{R_d}\right).
\]
The error tends to zero by assumption, and Proposition~\ref{prop:continuous} gives
\[
 \frac1d\log e_d\longrightarrow\log\frac{4\sqrt3}{9}.
\]
Exponentiating proves the theorem.
\end{proof}

\begin{proof}[Proof of Corollary~\ref{cor:construction}]
Every $z=(z_1,\ldots,z_{d(n)})\in\Lambda_{d(n),R_n}$ satisfies
$|z_i|\le R_n$ for every $1\le i\le d(n)$. Hence every coordinate of
$R_n\one_{d(n)}+z$ lies in $\{0,1,\ldots,2R_n\}$, and
$2R_n\le n-1$. Thus
\[
 A_n\subset\{0,1,\ldots,n-1\}^{d(n)}.
\]
Translation preserves cardinality and additive energy, so
\[
 |A_n|=|\Lambda_{d(n),R_n}|,
 \qquad
 E(A_n)=E(\Lambda_{d(n),R_n}).
\]

By Lemma~\ref{lem:count},
\[
 \log |A_n|
 =\log v_{d(n)}+d(n)\log R_n
  +O\left(\frac{(d(n))^{3/2}}{R_n}\right).
\]
Stirling's formula gives
\[
 \log v_{d(n)}
 =-\frac{d(n)}2\log d(n)+O(d(n)),
\]
while
\[
 \log R_n=\log n-\log 2+O(n^{-1}).
\]
Since $\log d(n)=o(\log n)$, we have $d(n)=n^{o(1)}$ and hence
$\sqrt{d(n)}/n\to0$. Therefore,
\[
 \log |A_n|
 =d(n)\log n-\frac{d(n)}2\log d(n)+O(d(n)).
\]
In particular,
\begin{equation}
 \label{eq:log-An-asymp}
 \log |A_n|=(1+o(1))d(n)\log n.
\end{equation}

Moreover, $R_n/\sqrt{d(n)}\to\infty$. Theorem~\ref{thm:lattice-ball},
together with \eqref{eq:log-An-asymp}, therefore implies
\[
 \log\frac{E(A_n)}{|A_n|^3}
 =d(n)\log\frac{4\sqrt3}{9}+o(d(n))
 =-d(n)\log\frac{3\sqrt3}{4}+o(d(n)).
\]
Equivalently,
\[
 \log E(A_n)
 =3\log|A_n|-d(n)\log\frac{3\sqrt3}{4}+o(d(n)).
\]
Finally,
\[
 \begin{aligned}
 \frac{\log E(A_n)}{\log |A_n|}
 &=3-\frac{d(n)\log(3\sqrt3/4)+o(d(n))}{\log |A_n|}\\
 &=3-\frac{\log(3\sqrt3/4)+o(1)}{\log n}\\
 &=3-\bigl(1+o(1)\bigr)\log_n\frac{3\sqrt3}{4}.
 \end{aligned}
\]
Exponentiating gives
\[
 E(A_n)
 =|A_n|^{\,3-(1+o(1))\log_n(3\sqrt3/4)}.
\]
\end{proof}

%
%

\end{document}